\documentclass[11pt]{article}
\usepackage{amsmath,amsfonts,latexsym,amssymb,amsthm}
\usepackage{url, verbatim}

\newcommand{\Q}{\mathbb{Q}}

\newcommand{\Z}{\mathbb{Z}}

\begin{document}

\title{Torsion of elliptic curves over cubic fields}
\author{Filip Najman}
\date{}
\maketitle
\begin{abstract}
Although it is not known which groups can appear as torsion groups of elliptic curves over cubic number fields, it is known which groups can appear for infinitely many non-isomorphic curves. We denote the set of these groups as $S$. In this paper we deal with three problems concerning the torsion of elliptic curves over cubic fields. First, we study the possible torsion groups of elliptic curves that appear over the field with smallest absolute value  of its discriminant and having Galois group $S_3$ and over the field with smallest absolute value of its discriminant and having Galois group $\Z/3\Z$. Secondly, for all except two groups $G\in S$, we find the field $K$ with smallest absolute value of its discriminant such that there exists an elliptic curve over $K$ having $G$ as torsion. Finally, for every $G\in S$ and every cubic field $K$ we determine whether there exists infinitely many non-isomorphic elliptic curves with torsion $G$.
\end{abstract}
\section{Introduction.}
The possible torsion groups of an elliptic curve over the rational numbers are known by a theorem of Mazur, who actually gave two different proofs of the theorem \cite{maz1} and \cite{maz2}. These groups are:
$$\Z/n\Z,\ n=1,\ldots, 10,12,$$
$$\Z/2\Z\oplus \Z/2n\Z,\ n=1,\ldots,4.$$

In a similar manner all the possible torsion groups over the collection of all quadratic fields were determined by Kenku and Momose \cite{km} and Kamienny \cite{Kam1}. The author determined the possible torsion groups over the quadratic cyclotomic fields $\Q(i)$ and $\Q(\sqrt{-3})$ \cite{Naj1, Naj2}, which are also the two quadratic fields with the smallest discriminant.

It was proven by Parent \cite{Par1, Par2} that $13$ is the largest prime that can divide the order of the torsion of an elliptic curve over a cubic field. Jeon, Kim and Schweizer \cite{jks} found all the groups that appear as torsion for infinitely many non-isomorphic curves over cubic fields. These are the following groups:

$$\Z/n\Z,\ n=1,\ldots, 16,18,20,$$
$$\Z/2\Z\oplus \Z/2n\Z,\ n=1,\ldots7$$

We denote the set of the above groups as $S$. Unfortunately, it is not known whether there are other groups that appear as torsion for only finitely many curves. For a cubic field $K$, we will denote by $T(K)$ the set of all groups that appear as torsion of an elliptic curve over $K$ but are not contained in the set $S$. Note that the union of all $T(K)$ over all cubic fields is still a finite set, and one could use \cite{Par3} to give an effective bound. But this bound would be huge and would not be of help for practical purposes.

Since over the rationals and over quadratic fields all the groups that appear as torsion, do so for infinitely many curves, the possibility that $T(K)$ is empty for all cubic fields is not unreasonable.

In this paper we focus more on the torsion of elliptic curves over a single cubic field (as opposed to looking at all cubic fields simultaneously). We deal with three problems:

1) Determine the possible torsion groups of an elliptic curve over the field with smallest discriminant and having Galois group $S_3$ and over the field with smallest discriminant and having Galois group $\Z/3\Z$.

2) Find for every group from $S$ the field with the smallest discriminant having it as a torsion of an elliptic curve,

3) Determine for how many non-isomorphic curves does each of the groups from $S$ appear as torsion for any fixed cubic field $K$.

Note that a similar problem to 1) was solved by the author in \cite{Naj1,Naj2} for quadratic fields. The analogues of 2) and 3) for quadratic fields were dealt with successfully by Kamienny and the author in \cite{kn}.

We succeed in 1) under the assumptions that $T(K)$ is empty for these fields $K$. We also eliminate some possibilities for $T(K)$.

We succeed in 2) for all the groups except for $\Z/20\Z$ and $\Z/2\Z \oplus \Z/14\Z$. The reason we fail for these groups is that the modular curves $X_1$ corresponding to these groups are non-hyperelliptic of genus 3 and 4, and thus it is extremely hard to determine the existence of a cubic point on them.

We solve problem 3) completely. When counting curves throughout this paper, we will always count them up to isomorphisms.

Although we follow the general strategy of \cite{kn} and \cite{Naj2} to handle these problems, since the problems are much harder over cubic than over quadratic fields, we will use various other methods, not previously used in \cite{kn,Naj2}.

\section{Torsion over cubic fields with smallest discriminant}
As we will be mentioning certain cubic fields many times in the remainder of the paper, for the convenience of the reader we list the first seven cubic fields (taken from \cite{jj}) ordered by ascending absolute value of the discriminant in the table below. In the table, $i$ is the position of the field in the table, $\Delta$ the discriminant of the field, $G$ its Galois group (the Galois group of its normal closure, to be precise) and in the last column in the table is the generating polynomial of the field.

\begin{center}
\begin{tabular}{|c|c|c|c|}
\hline
$i$ & $\Delta$ & $G$ & Polynomial\\
\hline
1 & $-23$ & $S_3$ & $x^3 - x^2 + 1$\\
\hline
2 & $-31$ & $S_3$ & $ x^3+x-1$\\
\hline
3 & $-44$ & $S_3$ & $x^3 - x^2 + x + 1$\\
\hline
4 & 49 & $\Z /3\Z$ & $x^3 - x^2 - 2x + 1$\\
\hline
5 & $-59$ & $S_3$ & $x^3 + 2x - 1$\\
\hline
6 & $-76$ & $S_3$ & $x^3 - 2x - 2$\\
\hline
7 & 81 & $\Z /3\Z$ & $x^3 - 3x - 1$\\
\hline

\end{tabular}
\medskip

\textbf{Table 1.}

\end{center}

We will denote the $i$-th field in the table as $K_i$. We will focus on two fields, $K_1$ and $K_4$. The field $K_1$ is the cubic field with the smallest discriminant, while $K_4$ is both the field with smallest discriminant having $\Z/3\Z$ as a Galois group and the totally real field with smallest discriminant. Note that both fields have class number 1 and that $K_4$ is the maximal real subfield of the cyclotomic field $\Q (\zeta_7)$. By $\alpha_i$ we denote the element generating $K_i$.

Let $K$ be a number field. Denote by $Y_1(m,n)$ the affine modular curve whose $K$-rational points classify isomorphism classes of triples $(E, P_m, P_n)$, where $E$ is an elliptic curve (over $K$) and $P_m$ and $P_n$ are torsion points (over $K$) which generate a subgroup isomorphic to $\Z/m \Z \oplus \Z/n \Z$. For simplicity,  we write $Y_1(n)$ instead of $Y_1(1,n)$. Let $X_1(m,n)$ be the compactification of the curve $Y_1(m,n)$ obtained by adjoining its cusps.

Similarly, let $Y_0(N)$ be the affine curve whose $K$-rational points classify isomorphism classes of pairs $(E,C)$, where $E$ is an elliptic curve and $C$ is a cyclic $Gal(\overline{K}/K)$-invariant subgroup of $E(\overline{K})$ of order $N$, or a \emph{$N$-cycle}. We obtain $X_0(N)$ by adjoining the cusps to $Y_0(N)$.

Nice models of the curves $X_0(n)$ and $X_1(n)$ can be found for example in \cite{Baa, Yang}, while the curves $X_1(2,10)$ and $X_1(2,12)$ can be found in \cite{Rab}.

We will use \emph{division polynomials} in many places in this paper. For a definition and more information on division polynomials see \cite{was}. We denote by $\psi_n$ the $n$-th division polynomial, which satisfies that, for a point $P$ on an elliptic curve in Weierstrass form, $\psi_n(x(P))=0$ if and only if $nP=0$. Note that for even $n$ one has to work with $\psi_n/\psi_2$ to get a polynomial only in one variable.

We do our rank and torsion computations on elliptic curves and Jacobians of genus 2 curves throughout this paper using MAGMA.

We first prove a useful lemma.

\newtheorem{tm}{Lemma}
\begin{tm}
If the torsion group of an elliptic curve $E$ over $\Q$ has a nontrivial 2-Sylow subgroup, then over any number field of odd degree the torsion of $E$ will have the same 2-Sylow subgroup as over $\Q$.
\label{lem1}
\end{tm}
\begin{proof}
Let $K$ be a number field of odd degree. If $E(\Q)\simeq \Z/2nZ$, then the rest of the 2-torsion is defined over a quadratic field and hence not over $K$. So if the $2$-Sylow group increases there must be a $K$-rational (but not $\Q$-rational) point $P$ such that $2P=Q$, where $Q\in E(\Q)$ is a nontrivial torsion point whose order is a power of 2. But for fixed $Q$ the equation $2P=Q$ has exactly 4 solutions and it is easy to see that the orbits under the action of $Gal(\overline{\Q}/\Q)$ can only have lengths of 2 or 4.
\end{proof}

\begin{comment}
\newtheorem{lem2}[tm]{Lemma}
\begin{lem2}
If an elliptic curve $E$ over $\Q$ has no 3-torsion, then over any number field $K$ of odd degree, $E(K)$ will have no 3-torsion.
\label{lem2}
\end{lem2}
\begin{proof}
The 3-division polynomial has degree $4$, so if it is irreducible over $\Q$ it is irreducible over any odd degree number field.
\end{proof}
\end{comment}

\newtheorem{lem3}[tm]{Theorem}
\begin{lem3}
Let $T$ be a torsion group from Mazur's theorem and $K$ a cubic number field. There exist infinitely many elliptic curves with torsion $T$ over $K$.
\label{lem3}
\end{lem3}
\begin{proof}
First note that by \cite[Theorem 3.2]{jks} there are infinitely many elliptic curves over $\Q$ with each of the torsion groups from Mazur's theorem.

Let $K$ be a fixed cubic field. Let $E$ be an elliptic curve defined over $\Q$, such that $E(\Q)_{tors}\simeq \Z/2\Z$. By Lemma \ref{lem1}, $E(K)_{tors}\simeq \Z /2n\Z$, where $n$ is odd. From \cite[Lemma 5.5]{mr} only finitely many quadratic twists of $E(K)$ have any odd-order torsion and (since twisting does not change the 2-torsion) hence all but finitely many twists will have torsion $\Z/2\Z$. In exactly the same manner one proves that there are infinitely many curves with $E(K)_{tors}\simeq \Z/2\Z \oplus \Z/2\Z$.

The only groups apart of the ones from Mazur's theorem that can a priori appear as the torsion of infinitely many elliptic curves over $K$ are the ones such that the corresponding modular curve has genus $\leq 1$ and that they appear on the list in \cite[Theorem 3.4]{jks}. One can see that these are the groups $\Z/11 \Z$, $\Z/ 14\Z$, $\Z /15 \Z$, $\Z /2\Z \oplus \Z /10\Z$ and $\Z/2\Z \oplus \Z/12\Z$.

Let $E$ be an elliptic curve over $\Q$ with torsion $T=\Z /2\Z \oplus \Z/ 6\Z$ or $\Z/2\Z \oplus \Z/8\Z$. Then by Lemma \ref{lem1}, the 2-Sylow subgroups of $E(\Q)$ and $E(K)$ are equal. Thus the torsion of $E(K)$ is larger than $E(\Q)$ for only finitely many rational elliptic curves satisfying $E(\Q)_{tors}\simeq T$.

 By invoking Lemma \ref{lem1} when needed, it also follows that there are only finitely many elliptic curves defined over $\Q$ with torsion $\Z/ n\Z$, where $n=6,8,9,10$ or $12$ whose torsion becomes larger in $K$.

 Take $E(t)$ to be the family of rational elliptic curves
 $$E(t):y^2+xy+(5t+3)y=x^3+(5t+3)x^2,$$
 $t\in \Z$, with 4-torsion over $\Q$. All elliptic curves in this family have good reduction at 5, and $E(t)(\mathbb F_5)\simeq \Z/4 \Z$ for all $t\in \Z$. This shows that $E(t)(\Q)$ cannot have any 8-torsion, 12-torsion or torsion containing $\Z /2\Z \oplus \Z /4\Z$, for all $t\in \Z$. As the 3-division polynomial of $E(t)$ is an irreducible degree 4 polynomial (over $\Q (t)$), by Hilbert's irreducibility theorem, there are infinitely many $t$ such that $E(t)(K)$ has no 3-torsion. Since by Lemma \ref{lem1}, the 2-Sylow subgroups of $E(t)(\Q)$ and $E(t)(K)$ are equal, we have proved that there are infinitely many $E(t)$ with torsion group $\Z /4\Z$ over $K$.

 Let $E(t)$ be the family of rational elliptic curves
 $$E(t):y^2+5txy+(5t-1)y=x^3+(5t-1)x^2,$$
$t\in \Z$, with 5-torsion over $\Q$. All elliptic curves in this family have good reduction at 5, and $E(t)(\mathbb F_5)\simeq \Z/5 \Z$ for all $t\in \Z$. This rules out the possibility of 10-torsion in $E(\Q)$. If $E(t)(K)$ had any 3-torsion, it would inject into the residue field of (a prime over) 5. But $|E(t)(\mathbb F_{125})|=140$, implying that there is no 3-torsion in the residue field of (a prime over) 5, whatever the splitting behavior of 5 in $K$. This implies that $E(t)(K)$ has no 15-torsion, for all $t\in \Z$.
If $E(t)$ is written in short Weierstrass form $y^2=x^3+a(t)x+b(t)$, then the discriminant of $x^3+a(t)x+b(t)$ is a degree 7 polynomial in $t$. This implies that there are infinitely many values $t$ such that this polynomial generates a totally real cubic field and infinitely many values for which the same polynomial generates a complex cubic field. This now implies that infinitely many $E(t)$ will not have any 2-torsion in $K$.

A similar argumentation proves that infinitely many rational elliptic curves $E$ with $E(\Q)_{tors}\simeq \Z /7\Z$ will have no 2-torsion in $K$ and thus have $E(K)_{tors}\simeq \Z / 7\Z$.

We take the families of elliptic curves with 3-torsion
$$E_i(t): y^2 + ((-1)^i10t+2)xy +(10t+1)y=x^3,$$
 where $i=1$ or $2$ and $t$ is a non-zero integer. Note that all curves from both families have good reduction at 5, and $E_i(t)(\mathbb F_5)\simeq \Z / 3\Z$, and hence $E_i(t)(\Q)$ has no 2-torsion or 9-torsion. Thus all curves from these two families have torsion group isomorphic to $\Z /3\Z$. If $K$ is complex, then we choose $E_2(t)$ and if $K$ is totally real, we choose $E_1(t)$. The discriminant of $E_i(t)$ will then be of opposite sign than the discriminant of $K$ (implying that there is no 2-torsion in $E_i(t)(K)$). All curves from both families have good reduction at (a prime over) 2, and have either 3 or 9 points in the residue field of (the prime over) 2. This rules out the possibility of 15-torsion in any $E_i(t)(K)$. By factoring the 9-division polynomial of $E_i(t)$ and taking out the factors belonging to the 3-division polynomial, we are left with factors of degree larger than $9$ (as polynomials over $\Q(t)$). This means that, by Hilbert's irreducibility theorem, for infinitely many values $t$, there will be no $9$-torsion in $E_i(t)(K)$.

 Let $E(t)$ be the family of rational elliptic curves with
 $$E(t):y^2+xy-((5t+2)^2-\frac{1}{16})y=x^3-((5t+2)^2-\frac{1}{16})x^2,$$
 where $t\in \Z$, containing $\Z /2\Z \oplus \Z /4\Z$ as a torsion subgroup over $\Q$. All curves in the family have good reduction at 5, and all curves have 8 points in $\mathbb F_5$. This proves that none of the curves have torsion $\Z /2\Z \oplus \Z /8\Z$ and that $\Z /2\Z \oplus \Z /4\Z$ is the whole torsion group over $\Q$. The 3-division polynomial is irreducible (over $\Q(t)$) and by Hilbert's irreducibility theorem, there are infinitely many curves $E(t)$ that have no 3-torsion over any cubic field. Finally, after applying Lemma \ref{lem1}, we have proved that there are infinitely many curves with torsion $\Z /2\Z \oplus \Z /4\Z$ over $K$.

 As we have already proved that there are infinitely many elliptic curves with odd torsion over $K$, from \cite[Lemma 5.5]{mr} we conclude that each of these curves has a twist with trivial torsion. Thus there are infinitely many elliptic curves with trivial torsion over $K$.
\end{proof}

\newtheorem{ttm}[tm]{Theorem}
\begin{ttm} Suppose $T(K_1)=\emptyset$. Then the torsion of an elliptic curve over $K_1$ is isomorphic to one of the groups
$$\Z/n\Z,\ n=1,\ldots 10, 12,$$
$$\Z/2\Z\oplus \Z/2n\Z,\ n=1,\ldots 4, 6.$$
All of the groups actually appear as a torsion of infinitely many curves over $K_1$.
\label{t1}
\end{ttm}
\begin{proof}
By Theorem \ref{lem3}, all the torsion groups from Mazur's theorem appear infinitely often over $K_1$.

The modular curve $X_1(11)$ is an elliptic curve with an affine model $y^2 - y = x^3 - x^2$. We compute that $X_1(11)(K_1)$ has rank 0 and $X_1(11)(K_1)\simeq X_1(11)(\Q) \simeq \Z/5\Z$, so all the points $X_1(11)(K_1)$ correspond to cusps, implying the non-existence of an elliptic curve with 11-torsion over $K_1$.

The modular curve $X_1(14)$ is an elliptic curve with an affine model $y^2 +xy+ y = x^3 - x$. We compute that $X_1(14)(K_1)$ has rank 0 and $X_1(14)(K_1)\simeq X_1(14)(\Q) \simeq \Z/6\Z$, so all the points $X_1(11)(K_1)$ correspond to cusps, implying the non-existence of an elliptic curve with 14-torsion over $K_1$.

The modular curve $X_1(15)$ is an elliptic curve with an affine model $y^2 +xy+ y = x^3+x^2$.  We compute that $X_1(15)(K_1)\simeq X_1(15)(\Q) \simeq \Z/4\Z$, the points on $X_1(15)(K_1)$ again being cusps, thus proving the non-existence of 15-torsion over $K_1$.

The modular curve $X_1(2,10)$ is an elliptic curve with an affine model $y^2 =x^3+x^2-x$.  We compute that $X_1(2,10)(K_1)\simeq X_1(2,10)(\Q) \simeq \Z/6\Z$, the points on $X_1(2,10)(K_1)$ being cusps, implying that there does not exist an elliptic curve with torsion $\Z/2\Z\oplus \Z/10\Z$ over $K_1$.

The modular curve $X_1(2,12)$ is an elliptic curve with an affine model $y^2 =x^3-x^2+x$. We compute $X_1(2,12)(K_1)\simeq \Z/4\Z \oplus \Z$, where $(2\alpha_1^2 - 1  ,2\alpha_1^2 - 2\alpha_1 - 3)$ is an element of infinite order. Since each point on $X_1(2,12)(K_1)$ corresponds to an isomorphism class of elliptic curves with torsion $\Z/2\Z\oplus \Z/12\Z$ over $K_1$, it is easy to see that there are infinitely many curves with this torsion.

The modular curve $X_1(13)$ is a hyperelliptic curve of genus 2 having an affine model $y^2=x^6-2x^5+x^4-2x^3+6x^2-4x+1$. Denote by $J_1(N)$ the Jacobian of $X_1(N)$. We embed the curve into its Jacobian
$J_1(13)(K_1)$, and compute that $J_1(13)(K_1) \simeq \Z/19\Z$. The rank of the Jacobian can be computed in MAGMA, while for the computation of the torsion, we use the MAGMA code of S. Siksek that can be found at \url{http://www.warwick.ac.uk/staff/S.Siksek/progs/chabnf/g2-jac.m} (although we could manage to compute the torsion by examining the Jacobian over finite fields, as in \cite{Naj2}). We then proceed to compute the fiber of the map $X_1(13)(K_1)\rightarrow J_1(13)(K_1)$ and thus find that all the points on $X_1(13)(K_1)$ are cusps.

The modular curve $X_1(16)$ is a hyperelliptic curve of genus 2 having an affine model $y^2=x^5 + 2x^4 + 2x^2 - x$. In a similar way as above we compute $J_1(16)(K_1) \simeq \Z/2\Z\oplus\Z/10\Z$, and find that all the points on $X_1(16)(K_1)$ are cusps.

The modular curve $X_1(18)$ is a hyperelliptic curve of genus 2 having an affine model $y^2=x^6+2x^5+5x^4+10x^3+10x^2+4x+1$. In a similar way as above we compute $J_1(18)(K_1) \simeq \Z/21\Z$, and find that all points on $X_1(18)(K_1)$ are cusps.

As there is no 14-torsion over $K_1$, there obviously does not exist a curve with torsion $\Z/2\Z\oplus\Z/14\Z$.

It remains to prove that there is no $20$-torsion. As we mentioned in the introduction, the methods used above are hard to use on $X_1(20)$ as it is a non-hyperelliptic curve of genus 3. We will instead work with $X_0(20)$, and prove a stronger statement, i.e. that there is no 20-cycle over $K_1$. The modular curve $X_0(20)$ is an elliptic curve with an affine model $y^2=(x+1)(x^2+4)$ (see \cite{Yang}), and we compute $X_0(20)(K_1)\simeq X_0(20)(\Q)\simeq \Z/6\Z$. It is  known (see \cite{ogg}) that $X_0(4p)$ has 6 rational cusps when $p$ is a prime, so we conclude that all the points on $X_0(20)(K_1)$ are cusps.
\end{proof}

In general, for any cubic field $K$, it is known that there are no points of order 55 \cite[Remark (2.3)]{km}, 27 or 64 \cite[Theorem B]{mom} over $K$.

\textbf{Remark 1.} One could also easily prove that there is also no 21-torsion over $K_1$ in a similar same way as it was proven that there is no 20-torsion. One finds that $X_0(21)(K_1)=X_0(21)(\Q)$, having 4 noncuspidal points, corresponding to the 4 rational curves having a 21-cycle, and none of them having $21$-torsion. Unfortunately, this is all we can say about $T(K_1)$.

\newtheorem{tm2}[tm]{Theorem}
\begin{tm2} Suppose $T(K_4)=\emptyset$. Then the torsion of an elliptic curve over $K_4$ is isomorphic to one of the groups
$$\Z/n\Z,\ n=1,\ldots 10, 12,13,14,18 $$
$$\Z/2\Z\oplus \Z/2n\Z,\ n=1,\ldots 4.$$
The groups that appear over $\Q$ appear infinitely many times as torsion, while all the other groups appear only finitely many times.
\label{t2}
\end{tm2}
\begin{proof}
We first compute in a similar manner as in the proof of Theorem \ref{t1} that $Y_1(11)(K_4)$ is empty.

Next we find that the rank of $X_1(14)(K_4)$ is 0 and the torsion is $\Z/18\Z$. These 18 points comprise all the cusps of $X_1(14)$ (6 rational over $\Q$ and 6 rational over $K_4$), and the 6 noncuspidal points generate exactly two curves (if $P$ is a point of order 14 on $E$, then $(E,\pm P)$, $(E,\pm 3P)$ and $(E,\pm 5P)$ are three distinct points on $X_1(14)$ corresponding to the same curve $E$), which are
\begin{equation}y^2 + \frac{9\alpha_4^2 - 13\alpha_4 + 1}{7}xy + \frac{8\alpha_4^2 - 4\alpha_4 - 19}{7}y = x^3 + \frac{8\alpha_4^2 - 4\alpha_4 - 19}{7}x^2 \label{14t}\end{equation}
and
\begin{equation}
y^2 + \frac{3\alpha_4^2 + 5\alpha_4 + 5}{7}xy + \frac{8\alpha_4^2 + 7\alpha_4 -
   4}{7}y = x^3 + \frac{8\alpha_4^2 + 7\alpha_4 - 4}{7}x^2,
\label{14t2}
\end{equation}
both with the 14-torsion point $(0,0)$. The curves (\ref{14t}) and (\ref{14t2}) have $j$-invariants $255^3$ and $-15^3$ respectively, and both are CM curves. Note that $X_1(14)$ is itself an elliptic curve, so we have simultaneously proved that there also exists an elliptic curve with 18-torsion over $K_4$! As the curves above are the only ones with 14-torsion, by checking that they do not have another 2-torsion point, we prove that there are no curves over $K_4$ with torsion $\Z/2\Z\oplus \Z/14\Z$. Both curves have rank 0, so over $K_4$  there are no elliptic curves having $\Z/14\Z\oplus \Z$ as a subgroup.

One finds in exactly the same way as in Theorem \ref{t1} that there are no elliptic curves over $K_4$ with torsion subgroups $\Z/15\Z$, $\Z/16\Z$, $\Z/2\Z\oplus \Z/10\Z$, $\Z/2\Z\oplus \Z/12\Z$.

As $X_1(4N)$ is a cover of $X_1(2,2N)$, and $Y_1(2,10)(K_4)=\emptyset$, one sees that $Y_1(20)(K_4)=\emptyset$.

We find noncuspidal points on $X_1(13)$ and construct from one of them the curve (which has rational $j$-invariant)
$$
y^2 + (4\alpha_4^2 - 2\alpha_4 - 8)xy + (20\alpha_4^2 - 11\alpha_4 - 45)y =
x^3 + (20\alpha_4^2 - 11\alpha_4 - 45)x^2,
$$
with the point $(0,0)$ of order 13.

Note that since $X_1(13)$ and $X_1(18)$ are curves of genus 2, by Faltings' theorem there are only finitely many points on them over $K_4$ and hence only finitely many elliptic curves with 13-torsion and 18-torsion.
\end{proof}

We can say more about $T(K_4)$ than we did for $T(K_1)$.

\newtheorem{tm3}[tm]{Proposition}
\begin{tm3}
The groups $\Z/21\Z$, $\Z/24\Z$, $\Z/28\Z$, $\Z/32\Z$, $\Z/35\Z$, $\Z/36\Z$, $\Z/49\Z, \Z/52\Z, \Z/2\Z \oplus \Z/26\Z$ are not in $T(K_4)$.
\label{t3}
\end{tm3}
\begin{proof}
One proves the non-existence of 21-torsion exactly the same way as over $K_1$. The non-existence of 24-torsion follows from $Y_1(2,12)=\emptyset$.

As proven in Theorem \ref{t2}, there are only two curves having $\Z/14\Z$ as a subgroup, and one easily checks that neither of them contains $\Z/28\Z$.

The non-existence of $N$-torsion for $N=32,36$ and $49$ is proven by showing that the elliptic curve $X_0(N)$ has rank 0 and checking that all the torsion points are actually cusps. For all the curves except $X_0(32)$, MAGMA easily computes the rank using 2-descent. For $X_0(32)$ it only gives an upper bound of 2. However we can compute that the analytic rank is equal to 0, and since $K_4$ is a totally real field of odd degree, \cite{zha} implies that the algebraic rank is also equal to 0.

Since $X_0(35)$ is a genus 3 (hyperelliptic) curve, it is better not to work with  $X_0(35)$ directly. Instead, we can redo the proof of Kubert \cite[Proposition IV.3.5.]{kub} over $K_4$. We work with $X_0(35)/w_5$, where $w_5$ is an Atkin-Lehner involution. There is a nice model of $X_0(35)/w_5$:
$$E:y^2+y=x^3+x^2+9x+1.$$
We compute that $E(K_4)=E(\Q)\simeq \Z/3\Z$, and as already shown in \cite{kub}, these 3 points come from cusps of $X_0(35)$, thus proving that there is no 35-cycle over $K_4$.

 By \cite{bcdt} for a rational elliptic curve $E$ of conductor $N$ there exists a morphism
$$X_0(N)\longrightarrow E$$
called \emph{modular parametrization}. To eliminate $\Z/52\Z$ and $\Z/2\Z \oplus \Z/26\Z$ we more or less follow the strategy of the proof of \cite[Theorem 4.2]{jm}. We easily compute that for the elliptic curve
$$ E: y^2 = x^3 + x - 10$$
with conductor 52 it holds $E(K_4)\simeq E(\Q) \simeq \Z/2\Z$ and that the degree of the modular parametrization is 3. Thus there are at most 6 points on $X_0(52)(K_4)$, but (see \cite{ogg}) $X_0(52)$ has 6 rational cusps, so $Y_0(52)(K_4)=\emptyset$, implying $\Z/52\Z, \Z/2\Z \oplus \Z/26\Z\not\in T(K_4)$.

\end{proof}

\textbf{Remark 2.} One can prove that if $\Z/25\Z\not\in T(K_7)$, then $T(K_7)=\emptyset$ in the following way. First, one uses the same methods as in the proofs of Theorems \ref{t1} and \ref{t2} to prove that there are no elliptic curves with points of order 11, 13, 14, 15, 16, 18 and that no curve contains $\Z/2\Z \oplus \Z/12\Z$ as a subgroup. We find that there are infinitely many curves containing $\Z/2\Z \oplus \Z/10\Z$ and are unfortunately unable to determine whether there exists a curve with 20-torsion. The non-existence of an elliptic curve having a point of order 21, 24, 35 or 49 is proved in exactly the same way as in the proof of Proposition \ref{t3}. Finally, we use the modular parametrization (of degree 4) $X_0(40)\rightarrow E$, where $E$ is the elliptic curve $y^2=x^3-107x-426$ with $E(K_7)=E(\Q)\simeq \Z /2\Z $, to prove the non-existence of a 40-cycle, ruling out $\Z/40\Z$ and $\Z/2\Z \oplus \Z/20\Z$ as possible subgroups.
Note that all other non-cyclic torsion groups are ruled out by the Weil pairing.

 For any cubic field $K$ it seems hardest to rule out the existence of 25-torsion in $T(K)$, as $X_0(25)$ has genus 0 (and infinitely many points) and $X_1(25)$ has genus 12. There is an intermediate curve $B$
$$X_1(25)\xrightarrow{ \ \ \ 2\ \ \ }B\xrightarrow{ \ \ \ 5\ \ \ }X_0(25)$$
(where the numbers above the arrows denote the degrees of the maps), which was used by Kubert \cite{kub} to prove the non-existence of rational 25-torsion and by Kenku \cite{ken} to show that there is no 25-torsion over quadratic fields. However this is still a curve of genus 4, and is thus not suited for the methods used in this paper.

\section{Cubic fields with smallest discriminant with a given torsion group appearing over them}

In this section we will find, for a given group $G\in S- \{\Z/20\Z, \Z/2\Z \oplus \Z/14\Z \}$ the cubic field $K$ with smallest absolute value of its discriminant such that $G$ appears as torsion of some elliptic curve $E(K)$. We will follow the same general strategy as Kamienny and the author \cite{kn}, by examining fields by ascending $|\Delta(K)|$ and for each field either finding an elliptic curve with given torsion or proving the non-existence of such a curve. For a group $G\in S,$ we denote that field  by $M(G)$.

\newtheorem{tm4}[tm]{Proposition}
\begin{tm4}
$M(\Z/2\Z \oplus \Z/12\Z)=K_1$.
\label{t4}
\end{tm4}
\begin{proof}
This follows directly from Theorem \ref{t1}.
\end{proof}

\newtheorem{tm5}[tm]{Proposition}
\begin{tm5}
$M(\Z/14\Z)=M(\Z/2\Z \oplus \Z/10\Z)=K_2$.
\label{t5}
\end{tm5}
\begin{proof}
From Theorem \ref{t1}, we see that $\Z/14\Z$ and $\Z/2\Z \oplus \Z/10\Z$ do not appear as torsion groups over $K_1$. We compute $rank(X_1(14)(K_2))=1$, and find the elliptic curve
$$y^2 + \frac{13\alpha_2^2 + 8\alpha_2 + 29}{9}xy + \frac{-35\alpha_2^2 -
    25\alpha_2 - 49}{27}y = x^3 + \frac{-35\alpha_2^2 -
    25\alpha_2 - 49}{27}x^2$$
with the point $(0,0)$ having order 14.

In a similar way we compute $rank(X_1(2,10)(K_2))=1$ and find an elliptic curve with
torsion $\Z/2\Z\oplus\Z/10\Z$, for example
$$y^2 + \frac{-\alpha_2^2 - 9\alpha_2 + 6}{11}xy + \frac{-84\alpha_2^2 -
    52\alpha_2 - 145}{121}y = x^3 + \frac{-84\alpha_2^2 -
    52\alpha_2 - 145}{121}x^2.$$
\end{proof}

\newtheorem{tm6}[tm]{Proposition}
\begin{tm6}
$M(\Z/11\Z)=M(\Z/13\Z)=M(\Z/15\Z)=K_3$.
\label{t6}
\end{tm6}
\begin{proof}
One easily computes that $X_1(11)(K_i)\simeq \Z/5\Z$ has only cusps for $i=1,2$ and $X_1(11)(K_3)\simeq \Z/10\Z$ (the rank of $X_1(11)$ is 0 over $K_3$). We obtain that there is exactly one (up to isomorphism) elliptic curve with 11-torsion over $K_3$, this curve being
\begin{equation}y^2 + (3\alpha_3^2 - 5\alpha_3 - 3)xy + (8\alpha_3^2 + 8\alpha_3 + 2)y =
x^3 + (2\alpha_3^2 - 10\alpha_3 - 6)x^2\label{11t}\end{equation}
Note that this curve has rank 0.

One can compute that $Y_1(15)(K_i)=\emptyset$ for $i=1,2$ and that the rank of $X_1(15)(K_3)$ is 1. We find an elliptic curve
$$y^2 + (\alpha_3^2 + \alpha_3 + 5)xy + (8\alpha_3^2 - 4\alpha_3 + 4)y = x^3 + (2\alpha_3^2 - \alpha_3 + 1)x^2$$
with 15-torsion.

We compute that $rank(J_1(13))=0$ over the fields with smaller discriminant. Then it is easy to check that there are no elliptic curves with 13-torsion over those fields. Over $K_3$ we find that $J_1(13)$ has positive rank, and by a simple search we find points on $Y_1(13)$, obtaining the elliptic curve
$$y^2 + (-\alpha_3^2 + 2)xy + (\alpha_3^2 - 2\alpha_3 + 1)y = x^3 + (\alpha_3^2
    - 2\alpha_3 + 1)x^2.$$
\end{proof}

\newtheorem{tm7}[tm]{Proposition}
\begin{tm7}
$M(\Z/18\Z)=K_4$.
\label{t7}
\end{tm7}
\begin{proof}
One computes that the rank of the $J_1(18)(K_i)$ is 0 and $Y_1(18)(K_i)=\emptyset$ for $i<4$.
In Theorem \ref{t2}, it was already proved that 18-torsion appears over $K_4$.
\end{proof}

\newtheorem{tm8}[tm]{Proposition}
\begin{tm8}
$M(\Z/16\Z)=K_5$.
\label{t8}
\end{tm8}
\begin{proof}
This is proved exactly as the previous proposition. We give a curve with 16-torsion over $K_5$:

$$ y^2 + (-\alpha_5 + 2)xy + (-\alpha_5^2 - 2\alpha_5 + 1)y = x^3 + (-\alpha_5^2
    - 2\alpha_5 + 1)x^2.$$
\end{proof}

\newtheorem{tm9}[tm]{Proposition}
\begin{tm9}
Let $M(\Z/20\Z)=L_1$ and $M(\Z/2\Z\oplus \Z/14\Z)=L_2$. Then $|\Delta(L_1)|\leq 316$ and $|\Delta(L_2)|\leq 961$.
\label{t9}
\end{tm9}
\begin{proof}
One finds an elliptic curve with 20-torsion over $K_{40}$, the 40-th field in the table \cite{jj}, generated by the polynomial $x^3 - x^2 - 4x + 2$ by putting $t=2$ into \cite[Table 1, $N=20$]{jky}. Note that the polynomials $4x^3+8x^2+x-2$ obtained from \cite[Table 1, $N=20$]{jky} and $x^3 - x^2 - 4x + 2$ (from \cite{jj}) differ but they generate the same field. The field $K_{40}$ is a $S_3$ cubic field.

We find an elliptic curve with torsion $\Z/2\Z\oplus \Z/14\Z$ over $K_{143}$, the $143$-th field from \cite{jj}, generated by $x^3 - x^2 - 10x + 8$, using \cite[Example 4.3]{jky}. The polynomials from \cite{jj} and \cite[Example 4.3]{jky} differ, but they generate the same field. The field $K_{143}$ is a cyclic cubic field.
\end{proof}

\section{The number of elliptic curves with given torsion over a number field}

In this section we examine the number (up to isomorphism) of elliptic curves over cubic number fields. We will ignore the groups that appear as torsion already over $\Q$ as by Theorem \ref{lem3} they appear as torsion over every cubic field infinitely many times.

The same problem (again ignoring the groups from Mazur's theorem) for quadratic fields was dealt with in \cite{kn}. It was shown that for some groups, if there is one curve with given torsion there are infinitely many, for some groups there are always finitely many, while for some groups, over some quadratic fields there will be finitely many, and over others infinitely many.

We obtain the following result for cubic fields.

\newtheorem{tm10}[tm]{Theorem}
\begin{tm10}
a) There are only finitely many elliptic curves with torsion $\Z/13\Z$, $\Z/16\Z$, $\Z/18\Z$, $\Z/20\Z$ and $\Z/2\Z \oplus \Z/14\Z$ over any fixed cubic field.\\
b) Over any cubic field there is either no curve with torsion $\Z/15\Z$, $\Z/2\Z \oplus \Z/10\Z$ and $\Z/2\Z \oplus \Z/12\Z$ or there are infinitely many.\\
c) There is exactly one curve with 11-torsion over $K_3$. This is the curve (\ref{11t}). Over any other cubic field there are either no elliptic curves with 11-torsion or there are infinitely many.\\
d) There are exactly two curves with 14-torsion over $K_4$. These are the curves (\ref{14t}) and (\ref{14t2}). Over any other cubic field there are either no or infinitely many such curves.\\
e) For each of the groups $\Z/11\Z$, $\Z/14\Z$, $\Z/15\Z$ and $\Z/2\Z \oplus \Z/12\Z$ there exist infinitely many cubic fields with Galois group $S_3$ over which there exist infinitely many elliptic curves with that torsion. There exist infinitely many totally real cubic fields over which there exist infinitely many elliptic curves with torsion $\Z/2\Z \oplus \Z/10\Z$.
\label{t10}
\end{tm10}
\begin{proof}
a) Since all the modular curves parameterizing these torsions are of genus $\geq 2$, the result follows trivially from Faltings' theorem. Actually, the statement is true over a number field of any degree.

b) The only way it is possible that there is a finite number of curves with a torsion group parameterized by a modular curve $X_1$ of genus 1 is that over some cubic field the torsion of $X_1(K)$ is larger than $X_1(\Q)$ and that $rank(X_1(K))=0$. We can find the possible candidates for this by examining the division polynomials of $X_1$.

By \cite{Par1,Par2} the only primes that can divide the order of the torsion are the primes up to and including 13.

We start with $X_1(15)$. Note that $X_1(15)(\Q)\simeq \Z/4\Z$. We easily compute that $\psi_n(X_1(15))$ for $n=3,5,7,11,13$ are all irreducible of degree larger than 3, and hence $\psi_n(X_1(15))$ will not have a root over any cubic field. By Lemma \ref{lem1} the 2-Sylow subgroups of $X_1(15)(\Q)$ and $X_1(15)(K)$ are equal for all cubic fields $K$.

We compute $X_1(2,10)(\Q)\simeq \Z/6\Z$ and $\psi_n(X_1(2,10))$ for $n=5,7,11,13$ are all irreducible of degree larger than 3, and hence $\psi_n(X_1(2,10))$ will not have a root over any cubic field. By Lemma \ref{lem1} there are no additional cubic points of order $2^n$. The polynomial $\psi_9(X_1(2,10))$ factors as $(x-1)(x^3 + 7/3x^2 + 1/3x + 1/3)f_9f_{27}$, where $f_9$ and $f_{27}$ are irreducible polynomials of degree 9 and 27. The points with $x=1$ correspond to the rational points of order 3. We see that there is a factor of degree 3 in the factorization implying that the $x$-coordinate of an additional element of order dividing 9 satisfies $x^3 + 7/3x^2 + 1/3x + 1/3$. We examine $X_1(2,10)(\Q(\alpha)),$ where $\alpha$ is the root of $x^3 + 7/3x^2 + 1/3x + 1/3$. We check that the torsion of $X_1(2,10)(\Q(\alpha))$ is still $\Z/6\Z$. This is because the $y$-coordinate of the point corresponding to this polynomial is not defined over $\Q(\alpha)$. It is in fact defined over $\Q(\alpha, \sqrt{-3})$, over which $X_1(2,10)$ has torsion $\Z/3\Z\oplus \Z/6\Z$. We conclude that there are no noncuspidal cubic torsion points on $X_1(2,10)$ that are not defined over $\Q$.

We compute $X_1(2,12)(\Q)\simeq \Z/4\Z$ and $\psi_n(X_1(2,12))$ for $n=3,5,7,11,13$ are all irreducible of degree larger than 3, and hence $\psi_n(X_1(2,12))$ will not have a root over any cubic field. By Lemma \ref{lem1} there are no cubic points of even order apart from the ones defined over $\Q$.

c) From the short Weierstrass form we see that $K_3$ is the only cubic field over which $X_1(11)$ has a 2-torsion point. Over any other cubic field, $X_1$ has odd torsion. As $X_1(11)$ has good reduction at (a prime over) 2, its whole torsion subgroup embeds into the residue field of (a prime over) 2. By the Hasse-Weil bound, an elliptic curve can have at most 13 points over that field, so only $\Z /5\Z\ (\simeq X_1(11)(\Q))$ is possible.

d) We compute $X_1(14)(\Q)\simeq \Z/6\Z$ and $\psi_n(X_1(14))$ for $n=5,7,11,13$ are all irreducible of degree larger than 3, and hence $\psi_n(X_1(14))$ will not have a root over any cubic field. By Lemma \ref{lem1} there are no additional cubic $2^n$-torsion points. The polynomial $\psi_9(X_1(14))$ factors as $x( x^3 - 2x^2 - x + 1)(x^3 + x^2/3 - x + 1)f_6f_{27}$, where $f_6$ and $f_{27}$ are irreducible polynomials of degree 6 and 27. Note that the first polynomial generates $K_4$ and, as shown in Theorem \ref{t2}, $X_1(14)(K_4)\simeq \Z/18\Z$ and the additional torsion generates just the curves (\ref{14t}) and (\ref{14t2}). Let $K$ be the field generated by $x^3 + x^2/3 - x + 1$. We check that although the $x$-coordinate of additional torsion points is defined over $K$, the $y$-coordinate is not. Actually, the situation is similar to $X_1(2,10)$, i.e. $X_1(14)(K(\sqrt{-3}))\simeq \Z/3\Z\oplus \Z/6\Z$.

e) We use the proof of \cite[Lemma 3.3. (b)]{jks}, where it is proved that if an elliptic curve $E$ is written in the form $Y^2=X^3+Ax+B$, $A,B \in \Z$, then one can construct an infinite sequence of primes $p_i$ such that the roots $\xi_i$ of $X^3+Ax+B-p_i^{-20}$ generate distinct cubic fields $\Q(\xi_i)$, such that $E(\Q(\xi_i))$ has positive rank. For the modular curve $X_1(2,10)$, the discriminant of $X^3+Ax+B$ is positive, so for all but finitely many $p_i$, the discriminant of $X^3+Ax+B-p_i^{-20}$ will also be positive. Hence $X_1(2,10)$ will have positive rank over infinitely many totally real cubic fields. On the other hand, for the remaining modular curves of genus 1, the discriminant of $X^3+Ax+B$ is negative, implying that for all but finitely many $p_i$, the discriminant of $X^3+Ax+B-p_i^{-20}$ will also be negative. Hence these modular curves will have positive rank over infinitely many cubic fields with Galois group $S_3$.

\end{proof}

\textbf{Acknowledgements.}
The author was supported by the National Foundation for Science, Higher Education and Technological Development of the Republic of Croatia. We are grateful to Andrej Dujella and Peter Stevenhagen for their helpful comments. We are greatly indebted to the anonymous referee whose numerous comments significantly improved this paper both in presentation and mathematical content.

\vspace{1cm}

\small{MATHEMATISCH INSTITUUT, P.O. BOX 9512, 2300 RA LEIDEN, THE NETHERLANDS}\\

\emph{E-mail:} fnajman@math.leidenuniv.nl\\

AND\\

\small{DEPARTMENT OF MATHEMATICS, UNIVERSITY OF ZAGREB, BIJENI\v CKA CESTA 30, 10000 ZAGREB, CROATIA}\\

\emph{E-mail:} fnajman@math.hr

\end{document}